\documentclass[11pt]{amsart}
\usepackage[T1]{fontenc}
\usepackage{lmodern}
\usepackage[margin=1.1in]{geometry}
\usepackage{amsmath,amssymb,mathtools}
\usepackage[expansion=false]{microtype}

\usepackage{xcolor} 

\usepackage[hidelinks]{hyperref}

\usepackage{dsfont}   
\newcommand{\one}{\mathds{1}}

\newcommand{\R}{\mathbb{R}}
\newcommand{\Q}{\mathbb{Q}}
\newcommand{\Z}{\mathbb{Z}}
\newcommand{\N}{\mathbb{N}}
\newcommand{\T}{\mathbb{T}}

\DeclareMathOperator{\ran}{ran}

\DeclareMathOperator{\supp}{supp}

\DeclareMathOperator{\spanop}{span}

\newtheorem{theorem}{Theorem}
\newtheorem{proposition}[theorem]{Proposition}
\newtheorem{lemma}[theorem]{Lemma}
\newtheorem{corollary}[theorem]{Corollary}
\theoremstyle{definition}

\newtheorem{remark}[theorem]{Remark}

\title{Countable unions of intervals without Riesz spectra}

\author[D. G. Lee]{Dae Gwan Lee}
\address{Department of Applied Mathematics, Pukyong National University, Busan 48513, Republic of Korea}
\email{daegwan@pknu.ac.kr, daegwans@gmail.com}
\date{\today}
\subjclass[2020]{42C15}

\keywords{Riesz spectra, exponential systems, Riesz bases, countable unions of intervals}

\begin{document}

\begin{abstract}
We prove that every set $S = \bigcup_{n = 0}^{\infty} [ t_{2 n + 1} , t_{2 n} )$, where $t_0 > t_1 > \cdots > 0$ and $\lim_{n \rightarrow \infty} t_n = 0$, admits no Riesz spectrum, that is, no system of exponentials is a Riesz basis for $L^2 ( S )$. Up to a reflection and a null set, this class contains the example of Kozma, Nitzan, and Olevskii \cite{KNO}. Moreover, combining with the result of Kozma and Nitzan \cite{KN}, we show that a bounded countable union of pairwise disjoint intervals of positive length whose endpoints have exactly one accumulation point admits a Riesz spectrum if and only if it is a finite union of intervals up to a null set.
\end{abstract}

\maketitle



\section{Introduction and main result}
\label{sec:intro}

For a countable set $\Lambda \subset \R$ and a measurable set $S \subset \R$ of finite positive measure, 
we write $E ( \Lambda; S ) = \{ e^{2 \pi i \lambda ( \cdot )} |_S : \lambda \in \Lambda \}$ for the corresponding system of complex exponentials in $L^2 ( S )$. 
We call $\Lambda$ its (\emph{frequency}) \emph{spectrum} and $S$ its (\emph{spatial}) \emph{domain}. 
We call $\Lambda$ a \emph{Riesz spectrum} for $S$ if $E ( \Lambda; S )$ is a Riesz basis. 
When the domain is clear from the context, we simply write $E ( \Lambda )$ in place of $E ( \Lambda; S )$. 
For general background on frames and Riesz bases, we refer to \cite{Ch16}.

Seip \cite{Seip95} constructed Riesz spectra for unions of two bounded intervals and for certain unions of more intervals.
A significant breakthrough was made by Kozma and Nitzan \cite{KN}, who proved that every finite union of bounded intervals admits a Riesz spectrum.
Caragea and Lee \cite{CL22} proved that if $I_\ell = [ a_\ell , b_\ell )$, $\ell = 1 , \ldots , N$, are disjoint intervals contained in $[ 0 , 1 )$ such that $1 , a_1 , \ldots , a_N , b_1 , \ldots , b_N$ are linearly independent over $\Q$,
then there exist pairwise disjoint subsets $\Lambda_1 , \ldots , \Lambda_N$ of $\Z$ such that $E ( \bigcup_{\ell \in J} \Lambda_\ell ; \bigcup_{\ell \in J} I_\ell )$ is a Riesz basis for every $J \subset \{ 1 , \ldots , N \}$.
For any finite partition of $[ 0 , 1 )$ into intervals $I_1 , \ldots , I_N$,
Pfander, Revay, and Walnut \cite{PRW24} showed that there exists a partition of $\Z$ into sets $\Lambda_1 , \ldots , \Lambda_N$ such that $E ( \bigcup_{\ell \in J} \Lambda_\ell ; \bigcup_{\ell \in J} I_\ell )$ is a Riesz basis whenever $J \subset \{ 1 , \ldots , N \}$ is a set of consecutive indices, where the intervals are enumerated from left to right. 
Lee, Pfander, and Walnut \cite{LPW23} studied Riesz bases of exponentials restricted to prescribed, possibly overlapping subsets of the ambient domain.
In \cite{Lee24}, the author derived some sufficient conditions for combining two exponential Riesz bases into a Riesz basis for the union of their domains.

Kozma, Nitzan, and Olevskii \cite{KNO} constructed a bounded countable union of intervals without a Riesz spectrum.
Their example is obtained by an inductive construction based on repeated subdivisions into rapidly increasing numbers of pieces, where the interval endpoints have a single accumulation point.
Recently, the nonexistence of Riesz spectra for balls in $\R^n$, $n \geq 2$, and for triangles in the plane was proved independently by Ortega-Cerd\`a \cite{OC} and Wan \cite{Wan}.


In this paper, we generalize the result of Kozma, Nitzan, and Olevskii \cite{KNO} to bounded sets that are countable unions of intervals separated by gaps, both of positive length, whose endpoints have a single accumulation point, using a method inspired by Ortega-Cerd\`a \cite{OC}.
We state our main result for the simple case in which the endpoints decrease to a single point; the general case is given in Corollary~\ref{cor:general}.

\begin{theorem}\label{thm:main}
Let
\begin{equation}\label{eq:set-S}
    S = \bigcup_{n = 0}^{\infty} [ t_{2 n + 1} , t_{2 n} ) , 
\end{equation}
where $t_0 > t_1 > \cdots > 0$ with $\lim_{n \rightarrow \infty} t_n = 0$.
Then $S$ admits no Riesz spectrum.
\end{theorem}

We would like to point out that the endpoints $t_n$ in \eqref{eq:set-S} may decrease to $0$ at any rate, and no comparison between the lengths of the intervals and gaps is needed. 
For example, choosing $t_n = 2^{- n}$ in \eqref{eq:set-S} gives the dyadic set
\begin{equation}\label{eq:dyadic}
    S_* = \bigcup_{n = 0}^{\infty} [ 4^{- n} / 2 , 4^{- n} ) \subset [ 0 , 1 ) ,
\end{equation}
which was the author's original motivating example (see Remark~\ref{rem:early-attempts} below).
Moreover, up to a reflection and a null set, the example of Kozma, Nitzan, and Olevskii \cite{KNO} is of the form \eqref{eq:set-S} (see Remark~\ref{rem:KNO} below).
The intervals in \eqref{eq:set-S} are chosen to be half-open, but including or excluding the endpoints does not affect the space $L^2 ( S )$.
More generally, we show in Corollary~\ref{cor:general} that a bounded countable union of pairwise disjoint intervals of positive length whose endpoints have exactly one accumulation point admits a Riesz spectrum only if it is a finite union of intervals up to a null set; 
the converse also holds by the theorem of Kozma and Nitzan \cite{KN} mentioned above.



Note that the characteristic function $\one_S$ of the set $S$ in \eqref{eq:set-S} 
has a right-hand limit at every point except $0$, where both values $0$ and $1$ occur on sets of positive measure in every right neighborhood.
We will show that for bounded sets, this property prevents the existence of a Riesz spectrum (see Proposition~\ref{prop:oscillation} for integer frequencies and Corollary~\ref{cor:real} for real frequencies); neither the restriction to a single exceptional point nor the boundedness can be dropped (see Remarks~\ref{rem:essential2} and~\ref{rem:unbounded} below). 

Our proof relies on a discrete form of the stationary projection method of Ortega-Cerd\`a \cite{OC}.
We first consider the case of integer frequencies. 
For an integer Riesz spectrum, the map taking coefficients to the corresponding Fourier series on $S$ is a bounded isomorphism (see \eqref{eq:RieszBasisSynthesisOperator} below).
Lemma~\ref{lem:projections} expresses this property as an operator norm distance strictly smaller than $1$ between two orthogonal projections.
Taking limits of integer translates of the spectrum and averaging over these translates produce a new Hilbert space on which the shift has a nonzero fixed vector, together with a family of orthogonal projections $P_t$, indexed by the points $t$ of the circle $\R / \Z$, and a single orthogonal projection $Q$ such that $\| P_t - Q \| \leq \rho < 1$ for almost every $t$ (see Lemma~\ref{lem:stationary}).
We let $t$ tend to $0$ from the right along two sequences, one through the intervals of $S$ and the other through the gaps.
Along each sequence, the projections $P_t$ converge strongly, and the two limiting projections agree on the orthogonal complement of the fixed vectors of the shift, whereas on the fixed vectors, one projection acts as the identity and the other as zero.
Consequently, their ranges are strictly nested, yet both limiting projections lie at a distance strictly less than $1$ from $Q$, which contradicts Lemma~\ref{lem:projections}.
In contrast to \cite{OC}, the projection $Q$ does not depend on $t$, and the exceptional point is approached from one side only: one-sided limits of $\one_S$ at all other points take the place of the boundary condition of \cite[Remark~6.4]{OC}, which cannot hold for a set with countable boundary (see Remark~\ref{rem:boundary}).
Finally, for arbitrary real frequencies, rounding to a fine lattice and rescaling yield an integer Riesz spectrum for a rescaled domain.

Section~\ref{sec:preliminaries} contains the preliminary results on orthogonal projections and perturbations of frequencies. 
In Section~\ref{sec:proof}, we construct a stationary representation and prove a one-sided oscillation criterion for integer frequencies. 
We then extend the criterion to real frequencies and prove Theorem~\ref{thm:main} and Corollary~\ref{cor:general}.
Section~\ref{sec:conclusion} contains concluding remarks, including a complementation principle and the author's earlier attempts on the dyadic set $S_*$ in \eqref{eq:dyadic}.

\section{Preliminaries}
\label{sec:preliminaries}

For a countable set $\Lambda \subset \R$ and a measurable set $S \subset \R$ of finite positive measure, 
the system $E ( \Lambda; S )$ is called a \emph{Riesz basis} (with \emph{Riesz bounds} $A$ and $B$) if it is complete in $L^2 ( S )$ and there exist constants $0 < A \leq B < \infty$ such that
\begin{equation}\label{eq:Riesz}
    A \sum_{\lambda \in \Lambda} | c_\lambda |^2 \leq \int_S \Big| \sum_{\lambda \in \Lambda} c_\lambda e^{2 \pi i \lambda x} \Big|^2 \, dx \leq B \sum_{\lambda \in \Lambda} | c_\lambda |^2
\end{equation}
for every finitely supported sequence $\{ c_\lambda \}_{\lambda \in \Lambda}$ (see \cite{Ch16}).
Equivalently, the synthesis operator 
\begin{equation}
\label{eq:RieszBasisSynthesisOperator}
\mathcal{T} : \; \{ c_\lambda \}_{\lambda \in \Lambda}  \; \longmapsto \;   \sum_{\lambda \in \Lambda} c_\lambda \, e^{2 \pi i \lambda ( \cdot ) } , 
\end{equation}
defined initially for finitely supported sequences,
extends to a bounded isomorphism from $\ell_2 ( \Lambda )$ onto $L^2 ( S )$. 
Indeed, the upper inequality of \eqref{eq:Riesz} gives the bounded extension of the synthesis operator $\mathcal{T}$, the lower inequality makes $\mathcal{T}$ injective with closed range, and the completeness of $E ( \Lambda; S )$ makes the range of $\mathcal{T}$ dense in $L^2 ( S )$, and hence equal to $L^2 ( S )$.

The Riesz inequality \eqref{eq:Riesz} implles that $\Lambda$ is separated, i.e., $\inf\{ |\lambda_1 - \lambda_2 | : \lambda_1 \neq \lambda_2 \in \Lambda \} > 0$ (see e.g., \cite[Proposition~3]{Lee24-Axioms}). If $\Lambda$ is separated and $S$ is bounded, then the upper inequality of \eqref{eq:Riesz} holds, i.e., $E ( \Lambda; S )$ is a Bessel sequence (see \cite[p.\,135, Theorem~4]{Yo01}).

For brevity, a spectrum contained in $\Z$ is called an \emph{integer spectrum}. 
By the bounds of a Riesz spectrum, we mean the Riesz bounds of its associated exponential system.

We will need the following lemma from \cite{OC}. 
For completeness, we provide a proof here.

\begin{lemma}[Lemmas~2.2 and~2.3 in \cite{OC}]
\label{lem:projections}
Let $P$ and $Q$ be orthogonal projections on a Hilbert space $\mathcal{H}$ with $\mathcal{P} := \ran P$ and $\mathcal{Q} := \ran Q$.
The restricted map $P |_{\mathcal{Q}} : \mathcal{Q} \rightarrow \mathcal{P}$ is a bounded isomorphism if and only if $\| P - Q \| < 1$.
In this case, $Q |_{\mathcal{P}} : \mathcal{P} \rightarrow \mathcal{Q}$ is also a bounded isomorphism.
Consequently, if $P$ and $\widetilde{P}$ are orthogonal projections with $\ran P \subsetneq \ran \widetilde{P}$, then there is no orthogonal projection $Q$ satisfying $\| P - Q \| < 1$ and $\| \widetilde{P} - Q \| < 1$.  
\end{lemma}

\begin{proof}
Note that for any $p \in \mathcal{P}$ and $q \in \mathcal{Q}$, 
\[
\langle P q, p \rangle_{\mathcal{P}} = \langle P q, p \rangle_{\mathcal{H}} 
= \langle q, p \rangle_{\mathcal{H}} 
= \langle q, Q p \rangle_{\mathcal{H}} 
= \langle q, Q p \rangle_{\mathcal{Q}} , 
\]
which shows that $Q |_{\mathcal{P}} : \mathcal{P} \rightarrow \mathcal{Q}$ is the adjoint of $P |_{\mathcal{Q}} : \mathcal{Q} \rightarrow \mathcal{P}$.
If $P |_{\mathcal{Q}}$ is a bounded isomorphism, 
then $P |_{\mathcal{Q}}$ and $(P |_{\mathcal{Q}})^*$ are are bounded below by a common constant $a \in ( 0 , 1 ]$, so that $\| ( I - Q ) P \| \leq \sqrt{1 - a^2}$ and $\| ( I - P ) Q \| \leq \sqrt{1 - a^2}$. 
Indeed, if $\mathcal{P} = \mathcal{Q} = \{ 0 \}$, then we may choose $a = 1$; otherwise, $P |_{\mathcal{Q}}$ and $(P |_{\mathcal{Q}})^*$ are both bounded below by $\| (P |_{\mathcal{Q}})^{- 1} \|^{- 1}$, so we may choose $a = \min \{ 1 , \| (P |_{\mathcal{Q}})^{- 1} \|^{- 1} \}$. 
Note that 
\[
( P - Q )^2
= P (I-Q) P + (I-P) Q (I-P) , 
\]
where $P (I-Q) P$ and $(I-P) Q (I-P)$ act on  $\mathcal{P}$ and $\mathcal{P}^{\perp}$, respectively.
Since $P - Q$ is self-adjoint, we obtain
\[
\begin{split}
      \| P - Q \|^2
      &= \max \{ \| P ( I - Q ) P \| , \| ( I - P ) Q ( I - P ) \| \} \\
      &= \max \{ \| P^* ( I - Q )^* ( I - Q ) P \| , \| ( I - P ) Q Q^* ( I - P )^* \| \} \\
      &= \max \{ \| ( I - Q ) P \|^2 , \| ( I - P ) Q \|^2 \} \\
      &\leq 1 - a^2 < 1 ,
\end{split}
\]
where we used that $\| C^* C \| = \| C C^* \| = \| C \|^2$ for every bounded operator $C$.
Conversely, if $\| P - Q \| = \rho < 1$, then both maps $P |_{\mathcal{Q}}$ and $Q |_{\mathcal{P}}$ are bounded below by $\sqrt{1 - \rho^2}$. 
Indeed, for any $q \in \mathcal{Q}$, we have $\| q - Pq \| = \| ( Q - P ) q \| \leq \rho \| q \|$, and since $Pq$ and $q - Pq$ are orthogonal, we get $\| Pq \|^2 = \| q \|^2 - \| q - Pq \|^2 \geq ( 1 - \rho^2 ) \| q \|^2$. 
Similarly, we get 
$\| Qp \|^2 \geq ( 1 - \rho^2 ) \| p \|^2$ for $p \in \mathcal{P}$.
This implies that $P |_{\mathcal{Q}}$ and $Q |_{\mathcal{P}}$ have closed ranges, which must be dense since their adjoints are injective.  
Therefore, both $P |_{\mathcal{Q}}$ and $Q |_{\mathcal{P}}$ are bounded isomorphisms.

Finally, let $P$ and $\widetilde{P}$ be orthogonal projections with $\mathcal{P}  := \ran P \subsetneq \widetilde{ \mathcal{P} } := \ran \widetilde{P}$, and suppose that there eixsts an orthogonal projection $Q$ satisfying $\| P - Q \| < 1$ and $\| \widetilde{P} - Q \| < 1$. 
Then the restricted maps $Q |_{\mathcal{P}} : \mathcal{P} \rightarrow \mathcal{Q}$ and $Q |_{\widetilde{\mathcal{P}}} : \widetilde{\mathcal{P}} \rightarrow \mathcal{Q}$ are bounded isomorphisms.
Since $Q |_{\mathcal{P}}$ is onto, there exist $h \in  \widetilde{ \mathcal{P} } \backslash \mathcal{P}$ and $g \in \mathcal{P}$ with $Q g = Q h$, which contradicts with the injectivity of $Q$ on $\widetilde{ \mathcal{P} }$.
\end{proof}

We will also need the following form of the Paley--Wiener stability theorem (see \cite[Section~2.3]{KN} and \cite[p.\,2008]{KNO}), which allows us to round real frequencies to points of a sufficiently fine lattice. 

\begin{lemma}\label{lem:rounding}
Let $S \subset [ 0 , 1 ]$ be a measurable set of positive measure, and assume that $\Lambda \subset \R$ is a Riesz spectrum for $S$ with bounds $A$ and $B$. 
Let $M \in \N$ satisfy
\begin{equation}\label{eq:rounding-condition}
    \sqrt{B} \big( e^{\pi / M} - 1 \big) < \sqrt{A} . 
\end{equation} 
For each $\lambda \in \Lambda$, choose a nearest point $\lambda' \in \frac{1}{M} \Z$.
Then the map $\lambda \mapsto \lambda'$ is injective, and its image $\Lambda' = \{ \lambda' : \lambda \in \Lambda \}$ is a Riesz spectrum for $S$.
\end{lemma}

\begin{proof}
Note that $| \lambda' - \lambda | \leq \frac{1}{2 M}$ by construction.
For $\lambda \in \Lambda$, let $\delta_\lambda = \lambda' - \lambda$.
Let $\mathcal{T} : \ell_2 ( \Lambda ) \rightarrow L^2 ( S )$ be the synthesis operator for $E ( \Lambda; S )$, let $D$ be the diagonal operator on $\ell_2 ( \Lambda )$ with entries $\delta_\lambda$, and let $M_x$ denote the multiplication by $x$ on $L^2 ( S )$.
Note that $\sqrt{A} \| c \| \leq \| \mathcal{T} c \| \leq \sqrt{B} \| c \|$ for $c \in \ell_2 ( \Lambda )$ by \eqref{eq:Riesz}, while $\| D \| \leq \frac{1}{2 M}$ and $\| M_x \| \leq 1$.  
For a finitely supported sequence $c = \{ c_\lambda \}_{\lambda \in \Lambda}$, let $\mathcal{T}' c = \sum_{\lambda \in \Lambda} c_\lambda e^{2 \pi i \lambda' ( \cdot )}$, where the points $\lambda'$ need not be distinct.
Expanding $e^{2 \pi i \lambda' x} = e^{2 \pi i \lambda x} \sum_{r = 0}^{\infty} \frac{( 2 \pi i \delta_\lambda x )^r}{r !}$, where the series converges uniformly for $x \in [ 0 , 1 ]$, we obtain
\[
    \mathcal{T} ' = \mathcal{T} + \sum_{r = 1}^{\infty} \frac{( 2 \pi i )^r}{r !} M_x^r \mathcal{T} D^r
\]
on finitely supported sequences.
Since the series on the right converges in operator norm, $\mathcal{T}'$ extends to a bounded operator from $\ell_2 ( \Lambda )$ to $L^2 ( S )$, and
\[
    \| \mathcal{T}' - \mathcal{T} \| \leq \sqrt{B} \left( e^{\pi / M} - 1 \right)  \; =:  \eta .
\]
By the triangle inequality, we have 
\[
      ( \sqrt{A} - \eta ) \| c \|
      \leq \| \mathcal{T}' c \|
      \leq ( \sqrt{B} + \eta ) \| c \|  ,
      \quad c \in \ell_2 ( \Lambda ) , 
\]
where $\sqrt{A} - \eta > 0$ by \eqref{eq:rounding-condition}.
In particular, $\mathcal{T}'$ is injective, and hence so is the map $\lambda \mapsto \lambda'$.
Indeed, if two distinct points of $\Lambda$ were rounded to the same point, then $\mathcal{T}'$ would vanish on the difference of the corresponding unit vectors of $\ell_2 ( \Lambda )$.
Identifying $\ell_2 ( \Lambda' )$ with $\ell_2 ( \Lambda )$ through the bijection $\lambda \mapsto \lambda'$, we see that $\mathcal{T}'$ is the synthesis operator for $E ( \Lambda' ; S )$, and the above estimate is the inequality \eqref{eq:Riesz} for $E ( \Lambda' ; S )$ with bounds $A' = ( \sqrt{A} - \eta )^2$ and $B' = ( \sqrt{B} + \eta )^2$.
To see that $\mathcal{T}' : \ell_2 ( \Lambda ) \rightarrow L^2 ( S )$ is onto, we observe that 
\[
\| \mathcal{T}^{-1} ( \mathcal{T}' - \mathcal{T} ) \| \leq \| \mathcal{T}^{- 1} \| \cdot \| \mathcal{T}' - \mathcal{T} \| 
\leq \frac{1}{\sqrt{A}} \cdot \eta < 1 , 
\]
so the Neumann series gives an inverse for $I + \mathcal{T}^{- 1} ( \mathcal{T}' - \mathcal{T} )$ on $\ell_2 ( \Lambda )$.  
Therefore, $\mathcal{T}' = \mathcal{T} ( I + \mathcal{T}^{- 1} ( \mathcal{T}' - \mathcal{T} ) )$ has the same range as $\mathcal{T}$, namely $L^2 ( S )$. 
Hence, we conclude that $\mathcal{T}'$ is a bounded isomorphism, that is, $\Lambda'$ is a Riesz spectrum for $S$.  
\end{proof}


\section{Proofs of Theorem~\ref{thm:main} and Corollary~\ref{cor:general}}
\label{sec:proof}

We identify $\T = \R / \Z$ with $[ 0 , 1 )$, where addition and one-sided limits are understood modulo~$1$, 
and use the normalized Lebesgue measure on $\T$. 
For a measurable set $S$, we denote by $\one_S$ the characteristic function of $S$.   

By a \emph{Borel representative} of $\one_S$, we mean a Borel $\{ 0 , 1 \}$-valued function that is equal to $\one_S$ almost everywhere.
The pointwise limits below refer to the chosen representative.
Note that changing a domain by a null set does not affect the Riesz basis property.


The spectral measure used below is associated with a unitary operator; it is distinct from a Riesz spectrum, which is a set of frequencies.
Recall that a unitary operator $V$ on a Hilbert space $\mathcal{K}$ has a spectral resolution
\[
      V = \int_\T e^{2 \pi i \theta} \, d \mathsf{E} ( \theta ) .
\]
Here, $\mathsf{E} ( \Delta )$ is an orthogonal projection for each Borel set $\Delta \subset \T$.
For $f \in \mathcal{K}$, let $\nu_f ( \Delta ) = \langle \mathsf{E} ( \Delta ) f , f \rangle$.
This defines a finite positive measure with $\nu_f ( \T ) = \| f \|^2$, and
\[
      \Big\| \int_\T b ( \theta ) \, d \mathsf{E} ( \theta ) f \Big\|^2
      = \int_\T | b ( \theta ) |^2 \, d \nu_f ( \theta )
\]
for every bounded Borel function $b$, which we call the \emph{symbol} of the operator $\int_\T b ( \theta ) \, d \mathsf{E} ( \theta )$.
In particular, a $\{ 0 , 1 \}$-valued symbol gives an orthogonal projection.
The range of $\mathsf{E} ( \{ 0 \} )$ is the subspace of fixed vectors of $V$. 
We will use strong operator convergence, meaning that $A_n f \rightarrow A f$ in norm for every fixed $f \in \mathcal{K}$.

\begin{lemma}[Stationary representation]
\label{lem:stationary}
Let $S \subset \T$ be a measurable set, and let $s : \T \rightarrow \{ 0 , 1 \}$ be a Borel representative of $\one_S$. 
Assume that $\Gamma \subset \Z$ is a Riesz spectrum for $S$. 
Then there exist a separable Hilbert space $\mathcal{K}$, a projection-valued spectral measure $\mathsf{E}$ on $\T$ acting on $\mathcal{K}$, an orthogonal projection $Q$ on $\mathcal{K}$, and a constant $0 \leq \rho < 1$ with the following properties: $\mathsf{E} ( \{ 0 \} ) \neq 0$, and the orthogonal projections
\begin{equation}\label{eq:def-P_t}
    P_t  := \int_\T s ( t - \theta ) \, d \mathsf{E} ( \theta ) , \quad t \in \T ,
\end{equation}
satisfy
\begin{equation}\label{eq:phase-gap}
    \| P_t - Q \|  \leq \rho   \quad \text{for all} \;\;  t \in G  ,
\end{equation}
where $G \subset \T$ is a measurable set of full Lebesgue measure.
Moreover, $t \mapsto P_t f$ is strongly measurable for every $f \in \mathcal{K}$.
\end{lemma}

Note that this lemma replaces the original frequency set with a shift-invariant probability space of its translates and their limits.
On the resulting space $\mathcal{K}$, the constant function is a nonzero fixed vector of the shift, which implies that $\mathsf{E} ( \{ 0 \} ) \neq 0$. 
Also, the Riesz basis assumption yields the bound $\| P_t - Q \|  \leq \rho < 1$ for almost every $t$. 
These are the two properties needed in Proposition~\ref{prop:oscillation} below.


\begin{proof}
\textbf{Step 1}.~Construction of the stationary space \\[1mm]
Let $\mathcal{F} : \ell_2 ( \Z ) \rightarrow L^2 ( \T )$ be the unitary Fourier series map, $\mathcal{F} c ( x ) = \sum_n c_n e^{2 \pi i n x}$, and let
\[
    P = \mathcal{F}^{- 1} M_s \mathcal{F} , \quad ( Q_\Gamma c )_n = \one_\Gamma ( n ) c_n .
\]
Here, $M_s$ denotes multiplication by $s$.
  Since $s = \one_S$ almost everywhere, $M_s$ is the orthogonal projection onto the
  subspace of functions supported in $S$.
  We identify $\ell_2 ( \Gamma )$ with $Q_\Gamma \ell_2 ( \Z )$ by extending
  coefficient sequences by zero.
  We also identify $L^2 ( S )$ with the corresponding subspace of $L^2 ( \T )$ by
  extending functions by zero.
  Under these identifications, the synthesis operator followed by $\mathcal{F}^{- 1}$
  is
\[
c \longmapsto \mathcal{F}^{- 1} ( \one_S \mathcal{F} c ) = Pc , 
\]
which shows that $\Gamma$ is a Riesz spectrum for $S$ if and only if the map
$P : Q_\Gamma \ell_2 ( \Z ) \rightarrow
  P \ell_2 ( \Z )$ is a bounded isomorphism.
Therefore, Lemma~\ref{lem:projections} gives
\begin{equation}\label{eq:original-gap}
    \| P - Q_\Gamma \| =: \rho < 1 .
\end{equation}
Let $\omega^\Gamma = ( \one_\Gamma ( n ) )_{n \in \Z}$, let $T$ be the shift $( T \omega )_n = \omega_{n + 1}$, and let
\[
    X = \overline{\{ T^k \omega^\Gamma : k \in \Z \}} \subset \{ 0 , 1 \}^{\Z} ,
\]
where the closure is taken in the product topology. 
Since the orbit $\{ T^k \omega^\Gamma : k \in \Z \}$ is invariant under the homeomorphism $T$ of $\{ 0 , 1 \}^{\Z}$, so is its closure $X$.
Therefore, $X$ is a closed, hence compact, subset of the compact metric space $\{ 0 , 1 \}^{\Z}$, and $T : X \rightarrow X$ is a homeomorphism.
For $\omega \in X$, let $Q_\omega$ be the coordinate projection with diagonal $( \omega_n )_n$.
The operator $P$ commutes with the bilateral shift $( U c )_n = c_{n - 1}$, since $\mathcal{F} U \mathcal{F}^{- 1}$ is multiplication by $e^{2 \pi i x}$, which commutes with $M_s$; moreover, $Q_{T^k \omega^\Gamma} = U^{- k} Q_\Gamma U^k$. 
Therefore, \eqref{eq:original-gap} holds with $Q_\Gamma$ replaced by $Q_{T^k \omega^\Gamma}$ for every $k \in \Z$. 
If $\omega^{( m )} \rightarrow \omega$ coordinatewise, then for every $c \in \ell_2 ( \Z
  )$,
  \[
      \| ( Q_{\omega^{( m )}} - Q_\omega ) c \|^2
      = \sum_{n \in \Z} | \omega_n^{( m )} - \omega_n |^2 | c_n |^2
      \longrightarrow 0
  \]
by dominated convergence.
Since every $\omega \in X$ is the coordinatewise limit of a sequence of shifts $\omega^{( m )} = T^{k_m} \omega^\Gamma$, we obtain
\[
      \| ( P - Q_\omega ) c \|
      = \lim_{m \rightarrow \infty} \| ( P - Q_{\omega^{( m )}} ) c \|
      \leq \rho \| c \|
      \quad \text{for all} \;\; c \in \ell_2 ( \Z ) .
\]
Consequently, we have 
\begin{equation}\label{eq:hull-gap}
    \| P - Q_\omega \| \leq \rho 
    \quad \text{for all}  \;\; \omega \in X  .
\end{equation}
Let \[
      \mu_N = \frac{1}{N} \sum_{k = 0}^{N - 1} \delta_{T^k \omega^\Gamma} ,
  \]
  where $\delta_\omega$ denotes the unit point mass at $\omega$. Since $X$ is compact and metrizable, a subsequence of $( \mu_N )_{N \in \N}$ converges weakly to a Borel
  probability measure $\mu$.
  For every continuous function $h$ on $X$, we have 
  \[
      \int_X ( h \circ T - h ) \, d \mu_N
      = \frac{ h ( T^N \omega^\Gamma ) - h ( \omega^\Gamma ) }{N}
      \longrightarrow 0 .
  \]
Taking the weak limit yields $\int_X h \circ T \, d\mu = \int_X h \, d\mu$ for every continuous function $h$. 
Consequently, we have $\mu ( T^{-1} B ) = \mu ( B )$ for every Borel set $B \subset X$.\footnote{In particular, the probability of observing any given finite pattern of zeros and ones is independent of the starting position. This shift invariance is what we mean by \emph{stationarity}.}
On the separable Hilbert space $\mathcal{K} = L^2 ( X , \mu )$, we define
\begin{equation}\label{eq:def-V-Q}
    V f = f \circ T , \quad Q = M_q , \quad q ( \omega ) = \omega_0 .
\end{equation} 
Then $V$ is unitary, $Q$ is an orthogonal projection, and the constant function $1$ is a nonzero fixed vector of $V$.
We write the spectral resolution of $V$ in additive notation as
\[
    V = \int_\T e^{2 \pi i \theta} \, d \mathsf{E} ( \theta ) .
\]
In particular, $\mathsf{E} ( \{ 0 \} ) \neq 0$ since $\mathsf{E} ( \{ 0 \} ) 1 = 1$.
For each $t \in \T$, let $B_t = \{ \theta \in \T : s ( t - \theta ) = 1 \}$.
Then $P_t$, defined by \eqref{eq:def-P_t}, can be written as $P_t = \mathsf{E} ( B_t )$, and hence $P_t$ is an orthogonal projection.
Note that $P_t h = s ( t ) h$ for every $h$ in the range of $\mathsf{E} ( \{ 0 \} )$ and every $t \in \T$, since on such vectors the spectral integral uses only the value of its integrand at $\theta = 0$.
In particular, we have $P_t 1 = s ( t ) 1$, which indicates the role of the fixed vectors of $V$ in Proposition~\ref{prop:oscillation}; on these vectors, $P_t$ acts as the identity if $s ( t ) = 1$, and as zero if $s ( t ) = 0$.
For $f,g \in \mathcal{K}$, let
$\nu_{f,g}(\Delta)
    = \langle \mathsf{E}(\Delta)f,g \rangle$.
Since $(t,\theta) \mapsto s(t-\theta)$ is a bounded Borel function, the function 
\[
    t \longmapsto \langle P_t f,g \rangle
      = \int_\T s(t-\theta)\,d\nu_{f,g}(\theta)
\]
is measurable, so the map $t \mapsto P_t f$ is weakly measurable.
Since $\mathcal{K}$ is separable, we conclude that $t \mapsto P_t f$ is strongly measurable for every $f \in \mathcal{K}$.

\bigskip

\noindent \textbf{Step 2}.~Conjugation to $L^2 ( \T ; \mathcal{K} )$ \\[1mm]
Consider the Hilbert space $L^2 ( X \times \Z , \mu \otimes \mathrm{counting} )$, with $P$ and $U$ acting in the integer coordinate, and let $( \widetilde{\mathsf{Q}} F ) ( \omega , n ) = \omega_n F ( \omega , n )$.
Since $\widetilde{\mathsf{Q}} F ( \omega , \cdot ) = Q_\omega F ( \omega , \cdot )$, \eqref{eq:hull-gap} gives
\[
    \| ( P - \widetilde{\mathsf{Q}} ) F \|^2 = \int_X \| ( P - Q_\omega ) F ( \omega , \cdot ) \|_{\ell_2 ( \Z )}^2 \, d \mu ( \omega ) \leq \rho^2 \| F \|^2 .
\]
The map
\[
( W F ) ( \omega , n ) = F ( T^{- n} \omega , n )
\]
is unitary by the $T$-invariance of $\mu$, and satisfies
\[
W \widetilde{\mathsf{Q}} W^{- 1} = Q , \quad W U W^{- 1} = V^{- 1} U ,
\]
where $Q$ and $V$ act in the $X$ coordinate. 
Indeed, $( W^{- 1} F ) ( \omega , n ) = F ( T^n \omega , n )$, and direct substitution gives
\[
\begin{split}
      ( W \widetilde{\mathsf{Q}} W^{- 1} F ) ( \omega , n )
      &= \omega_0 F ( \omega , n ) , \\
      ( W U W^{- 1} F ) ( \omega , n )
      &= F ( T^{- 1} \omega , n - 1 ) .
\end{split}
\]
Therefore, $W$ removes the dependence of the coordinate projection on $n$, while the shift acquires the factor $V^{- 1}$.
Fourier series in the integer variable define a unitary map
\[
    \mathcal{G} :
    L^2 ( X \times \Z , \mu \otimes \mathrm{counting} )
    \longrightarrow L^2 ( \T ; \mathcal{K} ) , 
\]
where $L^2 ( \T ; \mathcal{K} )$ denotes the space of strongly measurable functions $g : \T \rightarrow \mathcal{K}$ satisfying $\int_\T \| g ( t ) \|^2 \, dt < \infty$, identified when they agree almost everywhere. 
For functions $F \in L^2 ( X \times \Z , \mu \otimes \mathrm{counting} )$ with finite support in the second coordinate, this map is given by
\[
    ( \mathcal{G} F ) ( t ) ( \omega )
    = \sum_{n \in \Z} F ( \omega , n ) e^{2 \pi i n t} . 
\]
Applying Parseval's identity, we get  
\[
    \int_\T | ( \mathcal{G} F ) ( t ) ( \omega ) |^2 \, dt
    = \sum_{n \in \Z} | F ( \omega , n ) |^2  
    \quad \text{for a.e.} \;\; \omega \in X , 
\]
and integrating over $X$ and using $\mathcal{K} = L^2 ( X , \mu )$ gives 
\[
    \int_\T \| ( \mathcal{G} F ) ( t ) \|_{\mathcal{K}}^2 \, dt
    = \sum_{n \in \Z} \| F ( \cdot , n ) \|_{\mathcal{K}}^2
    = \| F \|_{L^2 ( X \times \Z )}^2 .
\]
For an arbitrary $F \in L^2 ( X \times \Z , \mu \otimes \mathrm{counting} )$, let $F_N$ be the truncation of $F$ to the coordinates with $| n | \leq N$. 
Since $F_N \rightarrow F$ in $L^2$, the sequence $( \mathcal{G} F_N )_N$ is Cauchy in $L^2 ( \T ; \mathcal{K} )$, so its limit defines $\mathcal{G} F$ and gives the unique isometric extension.
To see that $\mathcal{G}$ is onto, fix any $f \in \mathcal{K}$ and $n \in \Z$. 
By taking $F ( \cdot , n ) = f$ and all other coordinates zero, we have $( \mathcal{G} F ) ( t ) = e^{2 \pi i n t} f$.
By linearity, the range of $\mathcal{G}$ contains $t \mapsto p ( t ) f$ for every trigonometric polynomial $p$. 
Given a measurable set $J \subset \T$, choose trigonometric polynomials $p_j \rightarrow \one_J$ in $L^2 ( \T )$.
Then
\[
    \| p_j f - \one_J f \|_{L^2 ( \T ; \mathcal{K} )}
    = \| p_j - \one_J \|_{L^2 ( \T )} \, \| f \|_{\mathcal{K}}
    \longrightarrow 0 .
\]
Since $\mathcal{G}$ is an isometry, its range is closed, so it contains every $\one_J f$ and thus every $\mathcal{K}$-valued simple function. 
Since these simple functions are dense in $L^2 ( \T ; \mathcal{K} )$, we conclude that $\mathcal{G}$ is onto.

Under the map $\mathcal{G}$, the operator $U$ becomes multiplication by $e^{2 \pi i t}$, while $Q$ and $V$ act pointwise on the $\mathcal{K}$-valued functions. 
For a bounded Borel function $b$ on $\T$, we write $b [ U ]$ for the operator obtained from the Borel functional calculus of $U$ with symbol $e^{2 \pi i t} \mapsto b ( t )$, so that $P = s [ U ]$.
We will show that the transformed action of $P$ on $L^2 ( \T ; \mathcal{K} )$ is $F ( t ) \mapsto P_t F ( t )$.
For $k \in \Z$, conjugation by $W$ sends $U^k$ to $( V^{- 1} U )^k$.
After Fourier transformation in $n$, this operator acts at $t$ as
\[
e^{2 \pi i k t} V^{- k} = \int_\T e^{2 \pi i k ( t - \theta )} \, d \mathsf{E} ( \theta ) .
\]
This establishes the identity $( \mathcal{G} W b [ U ] W^{- 1} \mathcal{G}^{- 1} F ) ( t ) = \int_\T b ( t - \theta ) \, d \mathsf{E} ( \theta ) F ( t )$ for each monomial $b ( x ) = e^{2 \pi i k x}$ with $k \in \Z$.
For a trigonometric polynomial
$p ( x ) = \sum_{k = - N}^{N} a_k e^{2 \pi i k x}$,
linearity gives
\[
\begin{split}
\big( \mathcal{G} W p [ U ] W^{- 1} \mathcal{G}^{- 1} F \big) ( t )
&= \sum_{k = - N}^{N} a_k e^{2 \pi i k t} \, V^{- k} F ( t ) \\
&= \Big( \int_\T p ( t - \theta ) \, d \mathsf{E} ( \theta ) \Big) F ( t ) .
\end{split}
\]

\bigskip

\noindent \textbf{Step 3}.~Passage to Borel multipliers \\[1mm]
For fixed $t$, convergence in $L^2 ( \T )$ need not imply convergence of the corresponding spectral integrals, since $\mathsf{E}$ may charge points $\theta$ where the approximants evaluated at $t - \theta$ do not converge to $s ( t - \theta )$.
Therefore, we prove the multiplier identity on $L^2 ( \T ; \mathcal{K} )$ by integrating in $t$.
Let $s_j = \Phi_j * s$, where $\Phi_j$ is the Fej\'er kernel given by 
\[
      \Phi_j ( t ) = \frac{1}{j + 1}
      \Big| \sum_{m = 0}^j e^{2 \pi i m t} \Big|^2 .
\]
Then each $s_j$ is a trigonometric polynomial with $0 \leq s_j \leq 1$, since $\Phi_j \geq 0$ and $\int_0^1 \Phi_j (t) \, dt = 1$. 
The approximation property of the Fej\'er kernels gives $s_j \rightarrow s$ in $L^2 ( \T )$. 
Therefore, we have $s_j [ U ] \rightarrow s [ U ]$ strongly.
Indeed, on the Fourier side, this follows first for bounded functions from $L^2$ convergence, and then extends to all $L^2$ functions by density and uniform boundedness.
The same strong convergence holds after adjoining the $X$ coordinate and conjugating by the above unitary maps.
On the other hand, we have 
\begin{equation}\label{eq:borel-justification}
    \int_\T \Big\| \int_\T ( s_j ( t - \theta ) - s ( t - \theta ) ) \, d \mathsf{E} ( \theta ) f \Big\|^2 dt = \| s_j - s \|_{L^2 ( \T )}^2 \| f \|^2 ,
    \quad f \in \mathcal{K} . 
\end{equation}
To see this equality, let $d_j = s_j - s$.
The spectral theorem gives
\[
\Big\| \int_\T d_j ( t - \theta ) \, d\mathsf{E} ( \theta ) f  \Big\|^2
= \int_\T | d_j ( t - \theta ) |^2 \, d\nu_f ( \theta ) ,
\quad t \in \T , 
\]
and by integrating in $t$ and using Fubini's theorem, we have
\[
\int_\T  \Big\| \int_\T d_j ( t - \theta ) \, d\mathsf{E} ( \theta ) f  \Big\|^2   dt
= \int_\T \left( \int_\T | d_j ( t - \theta ) |^2 \, dt \right) d\nu_f ( \theta ) .
\]
Since $\int_\T | d_j ( t - \theta ) |^2 \, dt = \| d_j \|_{L^2 ( \T )}^2$ for every $\theta$, the last term reduces to $\| d_j \|_{L^2 ( \T )}^2 \, \nu_f ( \T ) = \| s_j - s \|_{L^2 ( \T )}^2 \| f \|^2$, as desired. 
Let $\mathcal{A}_j = \mathcal{G} W s_j [ U ] W^{- 1} \mathcal{G}^{- 1}$ be the transformed operator corresponding to $s_j$, so that 
\[
( \mathcal{A}_j F ) ( t )
= \int_\T s_j ( t - \theta ) \, d \mathsf{E} ( \theta ) F ( t ) 
\]
by Step~2 applied to the trigonometric polynomial $s_j$. 
We define $\mathcal{A} : L^2 ( \T ; \mathcal{K} ) \rightarrow L^2 ( \T ; \mathcal{K} )$ by
$( \mathcal{A} F ) ( t ) = P_t F ( t )$. 
To see that $\mathcal{A}$ is well-defined and bounded, note that for a simple function $F$, the map $t \mapsto P_t F ( t )$ is strongly measurable by Step~1. 
For a general $F \in L^2 ( \T ; \mathcal{K} )$, we can express $F$ as the almost everywhere limit of a sequence of simple functions $F_m$, and $\| P_t F_m ( t ) - P_t F ( t ) \| \leq \| F_m ( t ) - F ( t ) \|$ for every $t$, so that $t \mapsto P_t F ( t )$ is strongly measurable as well, with $\| \mathcal{A} F \| \leq \| F \|$ since each $P_t$ is an orthogonal projection.

We first show the convergence $\mathcal{A}_j F \rightarrow \mathcal{A} F$ for a simple function
\[
F ( t ) = \sum_{\ell = 1}^m \one_{J_\ell} ( t ) f_\ell ,
\]
where the sets $J_\ell \subset \T$ are measurable and pairwise
disjoint, and $f_\ell \in \mathcal{K}$.
By \eqref{eq:borel-justification}, we have
\[
\begin{split}
\| ( \mathcal{A}_j - \mathcal{A} ) F   \|_{L^2 ( \T ; \mathcal{K} )}^2
&=  \sum_{\ell = 1}^m \int_{J_\ell}   \Big\|  \int_\T \big( s_j ( t - \theta ) - s ( t - \theta ) \big) \, d \mathsf{E} ( \theta ) f_\ell  \Big\|_{\mathcal{K}}^2 \, dt \\
&\leq  \| s_j - s \|_{L^2 ( \T )}^2   \sum_{\ell = 1}^m \| f_\ell \|_{\mathcal{K}}^2
\longrightarrow 0 .
\end{split}
\]
Therefore, we have $\mathcal{A}_j F \rightarrow \mathcal{A} F$ for every simple $\mathcal{K}$-valued function $F$.

To pass to an arbitrary $F \in L^2 ( \T ; \mathcal{K} )$, note that $\| \mathcal{A}_j \| \leq 1$ and $\| \mathcal{A} \| \leq 1$, since $0 \leq s_j \leq 1$ and each $P_t$ is an orthogonal projection.
Given $\varepsilon > 0$, choose a simple function $H$ such that
$\| F - H \|_{L^2 ( \T ; \mathcal{K} )} < \varepsilon$.
Then
\[
\begin{split}
    \| ( \mathcal{A}_j - \mathcal{A} ) F  \|_{L^2 ( \T ; \mathcal{K} )}
    &\leq   2 \| F - H \|_{L^2 ( \T ; \mathcal{K} )} + \| ( \mathcal{A}_j - \mathcal{A} ) H  \|_{L^2 ( \T ; \mathcal{K} )} \\
    &<  2 \varepsilon  + \| ( \mathcal{A}_j - \mathcal{A} ) H \|_{L^2 ( \T ; \mathcal{K} )} .
\end{split}
\]
For this fixed $H$, the last term tends to zero as $j \rightarrow \infty$.
Since $\varepsilon > 0$ is arbitrary, we conclude that
$\mathcal{A}_j F \rightarrow \mathcal{A} F$ for every
$F \in L^2 ( \T ; \mathcal{K} )$.
In other words, the transformed operators $\mathcal{A}_j$ converge strongly to $\mathcal{A}$, given by $F ( t ) \mapsto P_t F ( t )$.
Since $s_j [ U ] \rightarrow s [ U ] = P$ strongly, the operators $\mathcal{A}_j$ also converge strongly to $\mathcal{G} W P W^{-1} \mathcal{G}^{-1}$.
By uniqueness of the strong operator limit, we obtain
$\mathcal{A}  = \mathcal{G} W P W^{-1} \mathcal{G}^{-1}$, 
which establishes the multiplier identity for $s$.

\bigskip

\noindent \textbf{Step 4}.~Transfer of the projection estimate \\[1mm]
Since the unitary map $\mathcal{G} W$ carries $P$ to $\mathcal{A}$ and $\widetilde{\mathsf{Q}}$ to the pointwise action of $Q$, the estimate $\| P - \widetilde{\mathsf{Q}} \| \leq \rho$ on $L^2 ( X \times \Z , \mu \otimes \mathrm{counting} )$ from the beginning of Step~2 becomes
\[
    \int_\T \| ( P_t - Q ) F ( t ) \|^2 \, dt \leq \rho^2 \int_\T \| F ( t ) \|^2 \, dt , 
    \quad  F \in L^2 ( \T ; \mathcal{K} ) .
\]
For fixed $f \in \mathcal{K}$, we can apply this estimate to $F ( t ) = \one_J ( t ) f$ for a measurable set $J  \subset \T$ to get 
  \[
      \int_J \| ( P_t - Q ) f \|^2 \, dt
      \leq \rho^2 | J | \| f \|^2 .
  \]
Since this holds for every measurable set $J  \subset \T$, we have $\| ( P_t - Q ) f \| \leq \rho \| f  \|$ for all $t$ outside a null set that may depend on $f$. 
Choose a countable dense subset of $\mathcal{K}$ and intersect the corresponding sets of full measure.
Then the resulting set $G$ still has full measure.
For each $t \in G$, continuity in $f$ extends the estimate from the dense subset to all of $\mathcal{K}$, so that \eqref{eq:phase-gap} holds. 
Note that changing $s$ on a null set may change the projections $P_t$ and the exceptional set, but it leaves the operator $\mathcal{A}$ on $L^2 ( \T ; \mathcal{K} )$ unchanged.
Indeed, let $\widetilde{s}$ be another Borel representative of $\one_S$, and let
\[
\widetilde{P}_t = \int_\T \widetilde{s} (t-\theta) \,d\mathsf{E} (\theta).
\]
Then the same calculation as in \eqref{eq:borel-justification} gives
\[
\int_\T \| (P_t-\widetilde{P}_t)f \|^2\,dt
= \|s-\widetilde{s}\|_{L^2(\T)}^2\|f\|^2
= 0, 
\quad f \in \mathcal{K}.
\]
By separability, $P_t = \widetilde{P}_t$ for all $t$ outside a single null set independent of $f$. 
\end{proof}

\begin{remark}\label{rem:boundary}
The criterion in \cite[Remark~6.4]{OC} is stated for bounded open sets in $\R^n$, $n \geq 2$. When transposed to a domain in $\R$, it requires a finite nonzero positive measure $\nu$ on its boundary $\Sigma \subset \R$ such that
$\nu ( \Sigma \cap ( \Sigma - \theta ) ) = 0$ for every $\theta \in \R \backslash \{ 0 \}$. 
This condition fails for the boundary $\Sigma = \{ 0 \} \cup \{ t_n : n \geq 0 \}$ of $S$ in Theorem~\ref{thm:main}.
Indeed, a finite nonzero positive measure on the countable set $\Sigma$ has an atom $x$.
Choose $y \in \Sigma \backslash \{ x \}$ and let $\theta = y - x$.
Then $x \in \Sigma \cap ( \Sigma - \theta )$, so this intersection has positive $\nu$-measure.
Instead, we approach the exceptional point $0$ from one side only, using one sequence in $S$ and another in its complement $S^c$. 
Since both sequences converge to $0$ from the right, their translated symbols have the same limit at every nonzero spectral point. 
Therefore, the two limiting projections differ by precisely the spectral projection at $0$.
\end{remark}


We first settle the case of integer frequencies by proving the following one-sided oscillation criterion.

\begin{proposition}\label{prop:oscillation}
Let $S \subset \T$ be a measurable set, and let $s : \T \rightarrow \{ 0 , 1 \}$ be a Borel representative of $\one_S$.
Assume that there is a point $a \in \T$ such that
\begin{enumerate}
\renewcommand{\labelenumi}{(\roman{enumi})}
\item for every $x \neq a$, the limit $\lim_{u \rightarrow 0^+} s ( x + u )$ exists;
\item for every sufficiently small $\varepsilon > 0$, the sets $\{ u \in ( 0 , \varepsilon ) : s ( a + u ) = j \}$, $j = 0 , 1$, have positive Lebesgue measure.
\end{enumerate}
Then $S$ admits no integer Riesz spectrum.
\end{proposition}

The idea of the proof is as follows.
Assuming that $S$ has an integer Riesz spectrum, Lemma~\ref{lem:stationary} provides the spectral measure $\mathsf{E}$ and the set $G$ of full measure. 
Condition (ii) allows us to choose two sequences in $G$ that approach $a$ from the right, one through $S$ and the other through its complement, while condition (i) makes the two limiting symbols agree at every $\theta \neq 0$.
At $\theta = 0$, the limiting symbols differ, and since $\mathsf{E} ( \{ 0 \} ) \neq 0$, the limiting projections have strictly nested ranges, which contradicts with Lemma~\ref{lem:projections}.

\begin{proof}[Proof of Proposition~\ref{prop:oscillation}]
Suppose to the contrary that $\Gamma \subset \Z$ is a Riesz spectrum for $S$.
Let $\mathcal{K}$, $\mathsf{E}$, $Q$, $\rho$, and $G$ be as in Lemma~\ref{lem:stationary}, together with the projections $P_t$ given by \eqref{eq:def-P_t}.

Since the sets in condition (ii) 
have positive measure and $G$ has full measure, for each $j \in \{ 0 , 1 \}$ we can choose a decreasing sequence of positive numbers $u_k^{( j )}$ tending to $0$ such that $s ( a + u_k^{( j )} ) = j$ and $a + u_k^{( j )} \in G$.
For any fixed $\theta \neq 0$, we have $a - \theta \neq a$ on $\T$, so condition (i) gives a right-hand limit of $s$ at $a - \theta$.
Since $s$ takes only the values $0$ and $1$, there exsits some $\delta_\theta > 0$ such that $s ( a - \theta + u )$ is constant for $0 < u < \delta_\theta$. 
Then both sequences $u_k^{( j )}$ eventually lie in the interval $( 0 , \delta_\theta )$, so that $\lim_{k \rightarrow \infty} s ( a + u_k^{( j )} - \theta ) = \lim_{u \rightarrow 0^+} s ( a - \theta + u )$ for $j = 0 , 1$. 
For $\theta = 0$, we have $s ( a + u_k^{( j )} ) = j$ for every $k$, by the choice of $u_k^{( j )}$. 
Define
\[
    r ( \theta ) =
    \begin{cases}
    \displaystyle \lim_{u \rightarrow 0^+} s ( a - \theta + u ) , & \theta \neq 0 , \\[3pt]
    0 , & \theta = 0 .
    \end{cases}
\]
The function $r$ is Borel, since away from $0$ it is the pointwise limit of the Borel functions $\theta \mapsto s ( a - \theta + 1 / k )$ as $k \rightarrow \infty$.
Since $r$ and $r + \one_{\{ 0 \}}$ take only the values $0$ and $1$, the operators 
\[
    P^{( 0 )} = \int_\T r ( \theta ) \, d \mathsf{E} ( \theta ) , \quad P^{( 1 )} = \int_\T \big( r ( \theta ) + \one_{\{ 0 \}} ( \theta ) \big) \, d \mathsf{E} ( \theta )
\]
are orthogonal projections.
Let $r_j = r + j \one_{\{ 0 \}}$ for $j = 0 , 1$. 
For every $f \in \mathcal{K}$, the spectral theorem gives
\[
      \| ( P_{a + u_k^{( j )}} - P^{( j )} ) f \|^2
      = \int_\T
      | s ( a + u_k^{( j )} - \theta ) - r_j ( \theta ) |^2
      \, d \nu_f ( \theta ) 
      \longrightarrow 0 ,
\]
where we used the dominated convergence theorem, since the integrand converges to $0$ at every $\theta$ and is bounded by $1$. 
Moreover, since $r ( 0 ) = 0$, we have  
\begin{equation}\label{eq:nested}
    P^{( 1 )} = P^{( 0 )} + \mathsf{E} ( \{ 0 \} ) , \quad P^{( 0 )} \mathsf{E} ( \{ 0 \} ) = 0 .
\end{equation}
For a nonzero $h \in \mathsf{E} ( \{ 0 \} ) \mathcal{K}$, we have $P^{( 0 )} h = 0$ and $P^{( 1 )} h = h$.
Together with \eqref{eq:nested}, this gives the strict inclusion $P^{( 0 )} \mathcal{K} \subsetneq P^{( 1 )} \mathcal{K}$.
For every $f \in \mathcal{K}$, the estimate \eqref{eq:phase-gap} gives
\[
      \| ( P^{( j )} - Q ) f \|
      = \lim_{k \rightarrow \infty} \| ( P_{a + u_k^{( j )}} - Q ) f \|
      \leq \rho \| f \| 
      \quad \text{for} \;\;  j = 0 , 1  .
\]
Therefore, we have $\| P^{( j )} - Q \| \leq \rho < 1$ for $j = 0 , 1$, which contradicts with Lemma~\ref{lem:projections}.
\end{proof}

\begin{remark}\label{rem:essential}
Condition (ii) in Proposition~\ref{prop:oscillation} cannot be replaced by the mere nonexistence of the right-hand limit of $s$ at $a$.
For example, the characteristic function of $[ 1 / 2 , 1 ) \cup \{ 2^{- n} : n \geq 2 \}$ has no right-hand limit at $0$, but this set coincides
almost everywhere with $[ 1 / 2 , 1 )$, 
whose $L^2$ space has the orthogonal basis $E ( 2 \Z )$.
\end{remark}

\begin{remark} 
\label{rem:essential2}
The conclusion of Proposition~\ref{prop:oscillation} can fail if condition (i) allows two exceptional points. 
For example, let $ \Omega = S_* \cup ( ( [ 0 , 1 ) \backslash S_* ) + 1 )$.
Then the map $\varphi : \Omega \rightarrow [ 0 , 1 )$ defined by 
\[
\varphi ( x ) =
\begin{cases}
       x , & x \in S_* , \\
       x - 1, & x \in ( [ 0 , 1 ) \backslash S_* ) + 1 , 
\end{cases}
\]
is a measure-preserving bijection. 
Since $x - \varphi ( x ) \in \Z$, we have $e^{2 \pi i k \varphi ( x )} = e^{2 \pi i k x}$ for every $k \in \Z$ and $x \in \Omega$, and therefore, $E ( \Z ; \Omega )$ is an orthonormal basis.  
Let $C = \frac{1}{2} \Omega \subset [ 0 , 1 )$ and $s = \one_C$.
By a dilation, $E ( 2 \Z ; C )$ is an orthogonal basis, where each element has the same $L^2$ norm $\frac{1}{\sqrt{2}}$.
Since exactly one of $t$ and $t + \frac{1}{2}$ belongs to $C$ for each $t \in [ 0 , \frac{1}{2} )$, we have 
\[
s ( t - \tfrac{1}{2} ) = 1 - s ( t ) ,
\quad t \in \T .
\]
Note that the function $s$ has right-hand limits at every point except $0$ and $\frac{1}{2}$. 
At each of the points $0$ and $\frac{1}{2}$, both values $0$ and $1$ occur on sets of positive measure in every right neighborhood, so these are the two exceptional points. 
However, $C$ admits the integer Riesz spectrum $2 \Z$, since $E ( 2 \Z ; C )$ is an orthogonal basis.

Furthermore, we can explain why the proof of Proposition~\ref{prop:oscillation} does not apply to this example.
For $\Gamma = 2 \Z$, the space $X$ in the proof of Lemma~\ref{lem:stationary} consists of the two sequences
\[
\omega^{(0)}_n = \one_{2 \Z} ( n ) , \quad
\omega^{(1)}_n = \one_{2 \Z + 1} ( n ) ,
\quad n \in \Z .
\]
Since the shift interchanges these sequences, the invariant probability measure assigns mass $\frac{1}{2}$ to each, 
i.e., $\mu = \frac{1}{2} \delta_{\omega^{(0)}} + \frac{1}{2} \delta_{\omega^{(1)}}$.
Let us identify $f \in \mathcal{K} = L^2 ( X , \mu )$ with its values $( z_0 , z_1 ) = ( f ( \omega^{(0)} ) , f ( \omega^{(1)} ) )$. 
Then $\| f \|^2 = \frac{1}{2} ( | z_0 |^2 + | z_1 |^2 )$, and the operators $V$ and $Q$ in \eqref{eq:def-V-Q} are given by 
$V ( z_0 , z_1 ) = ( z_1 , z_0 )$ and $Q ( z_0 , z_1 ) = ( z_0 , 0 )$.  
Note that $( 1 , 1 )$ and $( 1 , -1 )$ are eigenvectors of $V$ with eigenvalues $1$ and $-1$, respectively.
Since $e^{2 \pi i \cdot 0} = 1$ and $e^{2 \pi i (1/2)} = -1$, these eigenvalues correspond to the spectral points $0$ and $\frac{1}{2}$, with associated spectral projections 
\[
E_+ := \mathsf{E} ( \{ 0 \} )
= \tfrac{1}{2}
      \begin{pmatrix} 1 & 1 \\ 1 & 1 \end{pmatrix} ,
\quad
E_- := \mathsf{E} ( \{ \tfrac{1}{2} \} )
= \tfrac{1}{2}
      \begin{pmatrix} 1 & -1 \\ -1 & 1 \end{pmatrix} .
\]
Now, the formula \eqref{eq:def-P_t} for $P_t$ gives
\[
P_t = s ( t ) E_+ + s ( t - \tfrac{1}{2} ) E_-
= s ( t ) E_+ + ( 1 - s ( t ) ) E_- .
\]
As $t \rightarrow 0^+$ through points where $s ( t ) = 0$, the projections $P_t$ are equal to $E_-$.
On the other hand, as $t \rightarrow 0^+$ through points where $s ( t ) = 1$, the projections $P_t$ are equal to $E_+$. 
The two limiting projections $E_-$ and $E_+$ have orthogonal one-dimensional ranges, namely,
\[
\ran E_+  = \spanop \{(1,1)\} , 
\quad
\ran E_-  = \spanop \{(1,-1)\} , 
\]
so neither range contains the other.
Indeed, along these two sequences, the symbols $\theta \mapsto s ( t - \theta )$ take the values $0$ and $1$, respectively, at $\theta = 0$, and the values $1$ and $0$, respectively, at $\theta = \frac{1}{2}$.
Therefore, the limiting symbols also differ at the nonzero spectral point $\frac{1}{2}$. 
The two spectral components are exchanged, as described for the one-dimensional interval in \cite[Remark~6.3]{OC}, and the strict inclusion of ranges used in the proof of Proposition~\ref{prop:oscillation} fails.
\end{remark}



After normalizing a bounded domain so that $S \subset [ 0 , 1 ]$, Lemma~\ref{lem:rounding} allows us to round a real Riesz spectrum to a Riesz spectrum contained in $\frac{1}{M} \Z$ for some $M \in \N$.
Under the change of variables $x = M y$, we have $e^{2 \pi i (k/M) x} = e^{2 \pi i k y}$ for $k \in \Z$ and $y = \frac{1}{M} x \in \frac{1}{M} S$. 
This allows us to reduce the case of a real Riesz spectrum to an integer Riesz spectrum.

\begin{corollary}\label{cor:real}
Let $S \subset \R$ be a bounded measurable set of positive measure, and let $s : \R \rightarrow \{ 0 , 1 \}$ be a compactly supported Borel representative of $\one_S$.
Assume that for some $a \in \R$, the limit $\lim_{u \rightarrow 0^+} s ( x + u )$ exists for every $x \neq a$, and
\begin{equation}\label{eq:real-oscillation}
    0 < \int_0^\varepsilon s ( a + u ) \, du < \varepsilon
\end{equation}
for every sufficiently small $\varepsilon > 0$.
Then $S$ admits no Riesz spectrum.
The same conclusion holds if $\lim_{u \rightarrow 0^+} s ( x - u )$ exists for every $x \neq a$ and $0 < \int_0^\varepsilon s ( a - u ) \, du < \varepsilon$ for every sufficiently small $\varepsilon > 0$.
\end{corollary}

\begin{proof}
Since translations and positive dilations preserve the stated hypotheses and the existence of a Riesz spectrum, 
we may assume that $S \cup \supp ( s ) \cup \{ a \} \subset ( 0 , 1 )$.
Suppose to the contrary that $\Lambda$ is a Riesz spectrum for $S$ with bounds $A$ and $B$.
Choose a large $M \in \N$ so that $\sqrt{B} ( e^{\pi / M} - 1 ) < \sqrt{A}$.
By Lemma~\ref{lem:rounding}, there exists a Riesz spectrum $\Lambda' \subset \frac{1}{M} \Z$ for $S$.
Let $A'$ and $B'$ be Riesz bounds for $\Lambda'$.
Applying the change of variables $x = M y$ gives an integer Riesz spectrum $M \Lambda'$ for $\frac{1}{M} S$, with bounds $\frac{A'}{M}$ and $\frac{B'}{M}$.
Indeed, for any finitely supported sequence $\{ c_{\lambda'} \}_{\lambda' \in \Lambda'}$, we have
\[
\int_{\frac{1}{M} S}
\Big| \sum_{\lambda' \in \Lambda'} c_{\lambda'}  e^{2 \pi i M \lambda' y} \Big|^2 \, dy 
= \frac{1}{M} \int_S
\Big| \sum_{\lambda' \in \Lambda'} c_{\lambda'}  e^{2 \pi i \lambda' x} \Big|^2 \, dx .
\]
Since the map $f \mapsto \sqrt{M} f ( M \cdot )$ is unitary from $L^2 ( S )$ onto $L^2 ( \frac{1}{M} S )$, the completeness is also preserved.

We define $s_M ( y ) = s ( M y )$ for $y \in [ 0 , 1 )$, and extend $s_M$ periodically over $\R$.
Since $s$ is supported in $( 0 , 1 )$, this extension is zero near the point where the endpoints of $[ 0 , 1 )$ are identified.
It is a Borel representative of $\one_{\frac{1}{M} S}$ on $\T$, and has a right-hand limit at every point except $\frac{a}{M}$.
For sufficiently small $\varepsilon > 0$, we have 
\[
    0 < \int_0^\varepsilon s_M ( \tfrac{a}{M} + u ) \, du = \frac{1}{M} \int_0^{M \varepsilon} s ( a + v ) \, dv < \varepsilon .
\]
Therefore, we arrive at a contradiction by applying Proposition~\ref{prop:oscillation} at $\frac{a}{M}$. 
The assertion for left-hand limits follows by reflection.
\end{proof}

\begin{proof}[Proof of Theorem~\ref{thm:main}]
Note that $S$ is a bounded Borel set, with $| S | \geq t_0 - t_1 > 0$, and its characteristic function $\one_S$ is compactly supported. 
Away from $0$, the set of endpoints is locally finite,
so this function has a right-hand limit at every nonzero point.
Given $\varepsilon > 0$, we can choose $n \in \N$ such that $t_{2 n} < \varepsilon$.
Then $( t_{2 n + 1} , t_{2 n} ) \subset S \cap ( 0 , \varepsilon )$ and $( t_{2 n + 2} , t_{2 n + 1} ) \subset ( 0 , \varepsilon ) \backslash S$.
Since both intervals have positive length, we have 
  \[
      0 < \int_0^\varepsilon \one_S ( u ) \, du < \varepsilon .
  \]
Applying Corollary~\ref{cor:real} with $s = \one_S$ and $a = 0$, we conclude that $S$ admits no Riesz spectrum.
\end{proof}

The same argument applies to the general case described in the introduction, 
where the endpoints may accumulate at the exceptional point from both sides.
Combining this with the theorem of Kozma and Nitzan \cite{KN} yields the following characterization, which also allows the intervals to be adjacent.

\begin{corollary}\label{cor:general}
Let $S = \bigcup_{k = 1}^{\infty} I_k \subset \R$ be a bounded set, where $I_1 , I_2 , \ldots$ are pairwise disjoint intervals of positive length, 
and the set of endpoints of $I_k$ has exactly one accumulation point.
Then $S$ admits a Riesz spectrum if and only if $S$ is a finite union of intervals up to a null set. 
In particular, $S$ admits no Riesz spectrum if the intervals $I_k$ are separated by positive gaps.  
\end{corollary}

\begin{proof}
Let $\mathcal{D}$ be the set of endpoints of the intervals $I_k$, and let $a$ be its accumulation point. 
Note that $S$ is a bounded Borel set of positive measure, and its characteristic function $s = \one_S$ is compactly supported.
We observe that every open interval $J$ with $J \cap \mathcal{D} = \emptyset$ is either contained in one of the intervals $I_k$ or disjoint from $S$, so that $s$ is constant on $J$.
Indeed, if $J$ meets $I_k$ without being contained in it, then the connected set $J$ contains a boundary point of $I_k$, that is, a point of $\mathcal{D}$.

Let $x \neq a$.
Since $x$ is not an accumulation point of $\mathcal{D}$, there exists $\eta > 0$ such that neither $( x - \eta , x )$ nor $( x , x + \eta )$ meets $\mathcal{D}$.
By the observation above, $s$ is constant on each of these two intervals, so $s$ has a left-hand limit and a right-hand limit at $x$.
Consequently, if
\begin{equation}\label{eq:oscillation-at-a-right}
    0 < | S \cap ( a , a + \varepsilon ) | < \varepsilon \;\; \text{for every} \;\; \varepsilon > 0  
\end{equation}    
or
\begin{equation}\label{eq:oscillation-at-a-left}
    0 < | S \cap ( a - \varepsilon , a ) | < \varepsilon \;\; \text{for every} \;\; \varepsilon > 0 ,
\end{equation}
then $S$ admits no Riesz spectrum by Corollary~\ref{cor:real}, applied to $s$ at the point $a$. 

Assume that $S$ admits a Riesz spectrum, so that neither \eqref{eq:oscillation-at-a-right} nor
\eqref{eq:oscillation-at-a-left} holds. 
Then there exist $\varepsilon_1 , \varepsilon_2 > 0$ such that each of the sets $S \cap ( a , a + \varepsilon_1 )$ and $S \cap ( a - \varepsilon_2 , a )$ is a null set or has a null complement in the respective interval. 
Note that since $\mathcal{D}$ is bounded and $a$ is the only accumulation point of $\mathcal{D}$, the sets $\mathcal{D} \cap [ a + \varepsilon , \infty )$ and $\mathcal{D} \cap ( - \infty , a - \varepsilon ]$ are finite for every $\varepsilon > 0$.
Let $U = ( a - \varepsilon_2 , a + \varepsilon_1 )$.
Every interval $I_k$ that is not contained in $U$ has an endpoint in the finite set $\mathcal{D} \backslash U$, and each point is an endpoint of at most two of the intervals $I_k$, since these are pairwise disjoint and of positive length.
Therefore, all but finitely many of the intervals $I_k$ are contained in $U$, so that $S \backslash U$ is a finite union of intervals.
Since $S \cap U$ coincides, up to a null set, with the union of those of the two intervals $( a , a + \varepsilon_1 )$ and $( a - \varepsilon_2 , a )$ in which the complement of $S$ is a null set, it follows that $S$ is a finite union of intervals up to a null set.  

Conversely, if $S$ is a finite union of intervals up to a null set, 
then $S$ admits a Riesz spectrum by the result of Kozma and Nitzan \cite{KN}, since changing a domain by a null set does not affect the Riesz basis property.
This proves the first assertion.

Now, assume that the intervals $I_k$ are separated by positive gaps, so that the closed intervals $\overline{I_k}$ are pairwise disjoint. 
It is enough to show that either \eqref{eq:oscillation-at-a-right} or \eqref{eq:oscillation-at-a-left} holds.
Since $\mathcal{D}$ has an accumulation point, $\mathcal{D}$ is infinite, 
so at least one of the sets $\mathcal{D} \cap ( a , \infty )$ and $\mathcal{D} \cap ( - \infty , a )$ is infinite. 
We will assume that $\mathcal{D} \cap ( a , \infty )$ is infinite and show that \eqref{eq:oscillation-at-a-right} holds; 
if $\mathcal{D} \cap ( - \infty , a )$ is infinite, then \eqref{eq:oscillation-at-a-left} holds in a similar way.  

Fix $\varepsilon > 0$.
Then $\mathcal{D} \cap ( a , a + \varepsilon )$ is infinite, so that infinitely many of the intervals $I_k$ have an endpoint in $( a , a + \varepsilon )$, and each of these intervals meets $( a , a + \varepsilon )$.
If an interval meets $( a , a + \varepsilon )$ but is not contained in $( a , a + \varepsilon )$, 
then it contains at least one of the endpoints $a$ and $a + \varepsilon$.
Since the intervals $I_k$ are pairwise disjoint, at most one $I_k$ contains $a$, and at most one $I_k$ contains $a + \varepsilon$.
This implies that at most two of the intervals $I_k$ meet $( a , a + \varepsilon )$ without being contained in it.
Consequently, there exist indices $k \neq m$ such that $I_k$ and $I_m$ are contained in $( a , a + \varepsilon )$ and $\sup I_k \leq \inf I_m$.
Let $b = \sup I_k$.
Since the closed intervals $\overline{I_k}$ and $\overline{I_m}$ are disjoint, 
we have $a < b < \inf I_m < a + \varepsilon$.
Since $b \neq a$, there exists $0 < \eta < \inf I_m - b$ such that $( b , b + \eta ) \cap \mathcal{D} = \emptyset$.
The interval $( b , b + \eta )$ is not contained in $I_k$, and it is not contained in $I_j$ for any $j \neq k$, since otherwise $b$ would belong to the closures of both $I_k$ and $I_j$.
By the observation above, $( b , b + \eta )$ is disjoint from $S$.
Therefore, $( a , a + \varepsilon )$ contains the interval $I_k \subset S$ and the interval $( b , b + \eta ) \subset \R \backslash S$, both of positive length, which gives $0 < | S \cap ( a , a + \varepsilon ) | < \varepsilon$.
Since $\varepsilon > 0$ was arbitrary, the condition \eqref{eq:oscillation-at-a-right} holds.
\end{proof}


\begin{remark}\label{rem:unbounded}
For unbounded sets of finite measure, the one-sided limit and oscillation conditions in Corollary~\ref{cor:real} do not imply the nonexistence of Riesz spectra.
For example, consider 
\[
\Omega = S_* \cup \bigcup_{n = 0}^{\infty}  ( G_n + n + 2 ) 
\quad \text{with} \;\; G_n = [ 4^{-n-1} , 4^{-n}/2 ) ,
\]
where $S_*$ is the set given by \eqref{eq:dyadic}. 
Then $\Omega$ is an unbounded union of intervals separated by gaps of positive length, and its endpoints have $0$ as their only accumulation point in $\R$.
The characteristic function of $\Omega$ has a right-hand limit at every nonzero point, while both values $0$ and $1$ occur on sets of positive measure in every right neighborhood of $0$.
However, since
\[
( 0 , 1 ) \backslash S_* = \bigcup_{n = 0}^{\infty}  G_n ,
\]
the map $\varphi : \Omega \rightarrow ( 0 , 1 )$ defined by
\[
\varphi ( x ) =
\begin{cases}
       x , & x \in S_* , \\
       x - ( n + 2 ) , & x \in G_n + n + 2 , 
\end{cases}
\]
is a measure-preserving bijection. 
Since $x - \varphi ( x ) \in \Z$, we have $e^{2 \pi i k \varphi ( x )} = e^{2 \pi i k x}$ for every $k \in \Z$ and $x \in \Omega$, and therefore, $E ( \Z ; \Omega )$ is an orthonormal basis.  
\end{remark}

\begin{remark}\label{rem:KNO}
The example constructed in \cite[Theorem~2]{KNO} is a countable union of intervals alternating with gaps of positive length, whose endpoints increase to a single point. 
After a reflection, this set is of the form \eqref{eq:set-S} up to a null set, so Theorem~\ref{thm:main} gives an alternative proof of the nonexistence result.
The quantitative estimates of \cite{KNO} for finite-stage sets and the disk are not recovered here.
\end{remark}

\section{Concluding remarks}
\label{sec:conclusion}

Theorem~\ref{thm:main} provides a wide class of countable unions of intervals without Riesz spectra, 
in which the interval lengths and the gap lengths can be prescribed independently, subject only to positivity and summability.
The restriction to one exceptional point in Proposition~\ref{prop:oscillation} cannot simply be omitted.
Indeed, Remark~\ref{rem:essential2} provides a set with two exceptional points that admits an orthogonal exponential basis,
where the limiting symbols also differ at a nonzero spectral point, and the ranges of the limiting projections are no longer nested, so the argument used for Proposition~\ref{prop:oscillation} breaks down.
Therefore, one-sided oscillation at finitely many points is not, by itself, sufficient for the nonexistence of Riesz spectra.

We close with the complementation principle and the author's earlier attempts on the dyadic set $S_*$ in \eqref{eq:dyadic}.
The following lemma is essentially from \cite[Proposition~2.1]{MM09} and \cite[Proposition~5.4]{BCMS16}; 
see also \cite[Proposition~3]{CL22} and \cite[Theorem~12]{PRW24}.

\begin{lemma}\label{lem:complement}
Let $S \subset [ 0 , 1 )$ be a measurable set with $0 < | S | < 1$, and let $\Gamma \subset \Z$.
Then $\Gamma$ is a Riesz spectrum for $S$ if and only if $\Z \backslash \Gamma$ is a Riesz spectrum for $[ 0 , 1 ) \backslash S$.
In this case, the optimal lower Riesz bounds of $\Gamma$ and $\Z \backslash \Gamma$ are identical. 
\end{lemma}


\begin{remark}
\label{rem:early-attempts}
Consider the dyadic set $S_*$ in \eqref{eq:dyadic}, that is, the case $t_n = 2^{- n}$ of \eqref{eq:set-S}.
Around 2019, the author examined the set $S_*$ in an effort to disprove the existence of integer Riesz spectra by combining the complementation principle (Lemma~\ref{lem:complement}) with a dilation.
This was among the author's first attempts to find a bounded set without a Riesz spectrum.
The existence of such sets was subsequently established by Kozma, Nitzan, and Olevskii \cite{KNO}.

Note that the intervals of $S_*$ are pairwise disjoint, and $| S_* | = \frac{1}{2} \sum_{n \geq 0} 4^{- n} = \frac{2}{3}$.
An interesting property of $S_*$ is
\[
( 0 , 1 ) \backslash S_* = \tfrac{1}{2} S_* . 
\]
Assume that $\Gamma \subset \Z$ is a Riesz spectrum for $S_*$ with optimal lower Riesz bound $a > 0$.
By a dilation, $2 \Gamma$ is a Riesz spectrum for $\tfrac{1}{2} S_*$ with optimal lower bound $a / 2$.
Lemma~\ref{lem:complement} then gives another integer Riesz spectrum $\Z \backslash 2 \Gamma$ for $[ 0 , 1 ) \backslash \frac{1}{2} S_* = S_* \cup \{ 0 \}$, and hence for $S_*$, with the same lower bound $a / 2$. 
Therefore, the iteration
\[
      \Gamma_0 := \Gamma , \quad
      \Gamma_{j + 1} := \Z \backslash 2 \Gamma_j
      \quad \text{for} \;\; j = 0 , 1 , 2 , \ldots 
\]
produces Riesz spectra whose optimal lower bounds are $2^{- j} a$.
Since these bounds tend to zero, this iteration does not give a uniform lower estimate with which to pass to a limiting frequency set.
Therefore, it does not by itself prove the nonexistence of an integer Riesz spectrum for $S_*$.

However, there is a simple argument excluding the two sets
\[
\Lambda_0 := \bigcup_{j = 0}^{\infty} 4^j ( 2 \Z + 1 ) , \quad \Lambda_1 := \Lambda_0 \cup \{ 0 \} .
\]
Note that every nonzero integer can be written uniquely as $2^k m$ with $k \geq 0$ and $m$ odd. 
The set $\Lambda_0$ consists of the integers for which $k$ is even, and $2 \Lambda_0$ consists of those for which $k$ is odd. Therefore, we have $\Z \backslash 2 \Lambda_0 = \Lambda_0 \cup \{ 0 \} = \Lambda_1$. 
Since $2 \Lambda_1 = 2 \Lambda_0 \cup \{ 0 \}$, we also have $\Z \backslash 2 \Lambda_1 = \Lambda_0$. 
If either $\Lambda_0$ or $\Lambda_1$ were a Riesz spectrum for $S_*$, then both sets would be Riesz spectra for $S_*$ by Lemma~\ref{lem:complement} and a dilation. 
However, this is impossible since adding one element to a complete system cannot give a Riesz basis.
This argument excludes these two particular frequency sets, but does not treat an arbitrary integer spectrum. 
Theorem~\ref{thm:main} shows that $S_*$ has no Riesz spectrum with real frequencies. 
\end{remark}

\section*{Acknowledgments}

The author is supported by the National Research Foundation of Korea (NRF) grant funded by the Korean government (MSIT) (RS-2026-25598099). 
The author would like to thank Andrei Caragea, Goetz Pfander, and David Walnut for fruitful discussions on the set $S_*$.  
As discussed in Remark~\ref{rem:early-attempts}, this work began with an attempt by the author to disprove the existence of integer Riesz spectra for $S_*$. 
The author subsequently obtained the result in its present form, with the assistance of a large language model (ChatGPT) for mathematical discussions. 
The author has thoroughly evaluated and verified all resulting material and takes full responsibility for the paper's mathematical content.

\end{document}